\documentclass[11pt, a4paper, reqno]{amsart}
\usepackage{amsmath, amsthm, amsfonts, amssymb,mathtools}
\usepackage{a4wide}
\usepackage[latin1]{inputenc}

\usepackage{enumitem}

\DeclareMathOperator{\reg}{reg}

\DeclareMathOperator{\PSL}{PSL}

\newcommand{\N}{\mathbb{N}}
\newcommand{\Z}{\mathbb{Z}}
\newcommand{\R}{\mathbb{R}}
\newcommand{\C}{\mathbb{C}}
\newcommand{\Q}{\mathbb{Q}}
\renewcommand{\H}{\mathbb{H}}

\DeclareMathOperator{\sgn}{sgn}

\DeclareMathOperator{\e}{\mathfrak{e}}
\renewcommand{\pmod}[1]{\  \,  \left( \mathrm{mod} \,  #1 \right)}

\DeclareMathOperator{\Gr}{Gr}
\DeclareMathOperator{\CT}{CT}

\numberwithin{equation}{section}
	\newtheorem{theorem}{Theorem}[section]
	
	\newtheorem{proposition}[theorem]{Proposition}

	\theoremstyle{definition}

	\newtheorem{remark}[theorem]{Remark}

\author{Markus Schwagenscheidt}

\title{Cycle integrals of meromorphic modular forms and Siegel theta functions}

\begin{document} 

\begin{abstract}
	We study meromorphic modular forms associated with positive definite binary quadratic forms and their cycle integrals along closed geodesics in the modular curve. We show that suitable linear combinations of these meromorphic modular forms have rational cycle integrals. Along the way, we evaluate the cycle integrals of the Siegel theta function associated with an even lattice of signature $(1,2)$ in terms of Hecke's indefinite theta functions, which is of independent interest.
\end{abstract}

\maketitle

\section{Introduction}

Let $\mathcal{Q}_{d}$ be the set of all (positive definite if $d < 0$) integral binary quadratic forms of discriminant $d$. For $k \in \N$ with $k \geq 2$ we consider the functions
\[
f_{k,d}(z) = \frac{|d|^{k-1/2}}{\pi}\sum_{Q \in \mathcal{Q}_{d}}\frac{1}{Q(z,1)^{k}},
\]
which transform like modular forms of weight $2k$ for $\Gamma = \PSL_{2}(\Z)$. These functions were first studied by Zagier \cite{zagierdoinaganuma} for $d > 0$, in which case they are cusp forms, and by Bengoechea for $d < 0$, in which case they are meromorphic modular forms. We obtain a refinement $f_{k,P}$ of $f_{k,d}$ by summing only over quadratic forms in the equivalence class of a fixed form $P$. 

In \cite{alfesbringmannschwagenscheidt} we computed the cycle integrals of the meromorphic modular forms $f_{k,P}$ associated with positive definite quadratic forms $P$. These cycle integrals are defined by\footnote{If $f_{k,P}$ has poles on $S_A$ then the cycle integral can be defined using the Cauchy principal value as in \cite{loebrichschwagenscheidtcycleintegrals}.}
\[
\mathcal{C}_A(f_{k,P}) = \int_{\Gamma_A \backslash S_A}f_{k,P}(z)A(z,1)^{k-1}dz,
\]  
where $A = [a,b,c]$ is an indefinite integral binary quadratic form of positive non-square discriminant $D > 0$, $\Gamma_A$ denotes the stabilizer of $A$, and $S_A = \{z \in \H \, : \, a|z|^2 + b\Re(z) + c = 0\}$ is the geodesic corresponding to $A$. We let $\mathcal{C}_D(f_{k,P})$ be the $D$-th trace of cycle integrals, that is, the sum of $\mathcal{C}_A(f_{k,P})$ over all classes $A \in \mathcal{Q}_D /\Gamma$, and we let $M_{3/2-k}^!$ denote the space of weakly holomorphic modular forms of weight $3/2-k$ for $\Gamma_0(4)$ satisfying Kohnen's plus space condition $a_g(n) = 0$ unless $(-1)^{k+1} n \equiv 0,1 \pmod 4$. Then the main result of \cite{alfesbringmannschwagenscheidt} is as follows.

\begin{theorem}[Theorem~1.1 in \cite{alfesbringmannschwagenscheidt}]\label{theorem abs}
Let $k \geq 2$ be even and let $g \in M_{3/2-k}^!$. Suppose that the coefficients $a_g(-D)$ are rational for $D > 0$ and vanish if $D$ is a square. Then the linear combinations of traces of cycle integrals
\[
\sum_{D > 0}a_g(-D)\mathcal{C}_D(f_{k,P})
\]
of the individual meromomorphic modular forms $f_{k,P}$ are rational.
\end{theorem}
	
In \cite{loebrichschwagenscheidtcycleintegrals} we gave some refinements of this result and made a conjecture concerning the rationality of the individual cycle integrals of linear combinations of various $f_{k,d}$'s, which can be viewed as a dual version of Theorem~\ref{theorem abs} for odd $k$. Our goal is to prove this conjecture. 

\begin{theorem}[Conjecture 2.5 in \cite{loebrichschwagenscheidtcycleintegrals}]\label{main theorem intro}
	Let $k\geq 3$ be odd and let $g \in M_{3/2-k}^!$. Suppose that the coefficients $a_g(d)$ for $d < 0$ are rational. Then the individual cycle integrals
	\[
	\mathcal{C}_A\left(\sum_{d < 0}a_g(d)f_{k,d}(z)\right)
	\]
	of linear combinations of the meromorphic $f_{k,d}$ are rational.
\end{theorem}

In \cite{loebrichschwagenscheidtcycleintegrals} this result was stated as a conjecture for higher level, even and odd $k$, and twisted versions of $f_{k,d}$. Our arguments also work in this general setup, see Section~\ref{section generalizations}.

The proof of Theorem~\ref{main theorem intro} follows a method of Bruinier, Ehlen, and Yang (see the proof of \cite[Theorem~5.4]{bruinierehlenyang}). We use the well-known fact that $f_{k,d}$ can be written as a regularized theta lift. Then we interchange the cycle integral with the regularized integral, which leads us to study the cycle integrals of the Siegel theta function $\Theta_L(\tau,z)$ associated with an even lattice $L$ of signature $(1,2)$. The evaluation of these cycle integrals is the main novelty of this work. Let us describe the result in more detail.

The indefinite quadratic form $A$ defines a vector of negative length in the lattice $L$, and hence yields sublattices $I = L \cap (\Q A)^\perp$ and $N = L \cap (\Q A)$ of signature $(1,1)$ and $(0,1)$, respectively. Hecke \cite{hecke} constructed an indefinite theta series $\vartheta_I(\tau)$ associated with the lattice $I$, which is a cusp form of weight $1$ for the Weil representation $\rho_I$, see \eqref{eq hecke theta series}. Moreover, there is a classical holomorphic weight $3/2$ theta series $\Theta_{3/2,N}(\tau)$ associated to the unary lattice $N$, see \eqref{eq holomorphic theta function}.

\begin{theorem}
	Suppose that $L = I \oplus N$. Then we have
	\[
	\mathcal{C}_A\left(\frac{\partial}{\partial z}\Theta_L(\tau,z) \right) \stackrel{\cdot}{=} \vartheta_I(\tau) \otimes \overline{\Theta_{3/2,N}(\tau)}v^{3/2}.
	\]
\end{theorem}

Here $\stackrel{\cdot}{=}$ means equality up to an explicit non-zero constant factor. For the precise version of this result we refer to Theorem~\ref{theorem cycle integral siegel theta}. The idea of the proof is that $\Theta_L(\tau,z)$ essentially splits as a product of two Siegel theta functions $\Theta_I(\tau,z)$ and $\Theta_{N}(\tau)$ along the geodesic $S_A$, and the cycle integral of $\Theta_I(\tau,z)$ yields the Hecke theta series $\vartheta_I(\tau)$.

Combining the above evaluation of the cycle integral of the Siegel theta function with the theta lift realization of $f_{k,d}$, we obtain a representation of $\mathcal{C}_A(f_{k,d})$ as a regularized Petersson inner product, which can in turn be evaluated using the theory of harmonic Maass forms. This gives an explicit formula for the cycle integrals appearing in Theorem~\ref{main theorem intro} (see Theorem~\ref{main theorem}), from which their rationality can easily be deduced.

This paper is organized as follows. In Section~\ref{section theta lifts} we recall the theta lift realization of $f_{k,d}$. In Section~\ref{section cycle integral siegel theta} we evaluate the cycle integrals of the Siegel theta function, which is the technical heart of this work. In Section~\ref{section cycle integral meromorphic} we put everything together and deduce an explicit rational formula for the cycle integrals of $f_{k,d}$, which also implies Theorem~\ref{main theorem intro}. Finally, in Section~\ref{section generalizations} we sketch the proof of a higher level version of Theorem~\ref{main theorem intro}.

\subsection*{Acknowledments} The author was supported by SNF project PZ00P2\_202210.

\section{Meromorphic modular forms as theta lifts}\label{section theta lifts}

In this section we recall the realization of the meromophic modular form $f_{k,d}$ as a theta lift of a harmonic Maass form, following \cite{alfesbruinierschwagenscheidt}. We consider the even lattice
\[
L = \left\{X=\begin{pmatrix}b & c \\ -a & -b \end{pmatrix}: a,b,c \in \Z \right\}
\]
with the quadratic form $q(X) = \det(X)$. It has signature $(1,2)$, and its dual lattice is given
\[
L' = \left\{X=\begin{pmatrix}b/2 & c \\ -a & -b/2 \end{pmatrix}: a,b,c \in \Z \right\}.
\]
The group $\Gamma = \PSL_{2}(\Z)$ acts on $L$ via $g.X = gXg^{-1}$. We can view $L'$ as the set of integral binary quadratic forms $[a,b,c]$, where $-4q(X)=b^2-4ac$ corresponds to the discriminant.

Let $\Gr(L)$ be the Grassmannian of positive definite lines in $V(\R) = L \otimes \R$. We can identify $\Gr(L)$ with $\H$ by sending $z = x+iy \in \H$ to the positive line generated by
\[
X(z) = \frac{1}{\sqrt{2}y}\begin{pmatrix}-x & |z|^2 \\ -1 & x \end{pmatrix}.
\]
This bijection is compatible with the actions of $\PSL_2(\R)$, in the sense that $g.X(z) = X(gz)$ for any $g \in \PSL_2(\R)$. We also consider the vectors
\[
 U_1(z) = \frac{1}{\sqrt{2}y}\begin{pmatrix}x & -x^2+y^2 \\ 1 & -x \end{pmatrix}, \qquad U_2(z) = \frac{1}{\sqrt{2}y}\begin{pmatrix}y & -2xy \\ 0 & -y \end{pmatrix},
\]
which, together with $X(z)$, form an orthogonal basis of $V(\R)$. For each fixed $z = x+iy \in \H$ we define the polynomials
\begin{align*}
Q_X(z) &= -\sqrt{2}y(X,U_1(z)+iU_2(z)) = az^2 + bz + c, \\
p_X(z) &= \sqrt{2}(X,X(z)) = \frac{1}{y}(a|z|^2 + bx + c),
\end{align*}
on $V(\R)$. They satisfy the invariances
\begin{align}\label{eq invariances}
	p_{X}(g z) = p_{g^{-1}.X}(z), \qquad Q_{X}(gz) = (cz+d)^{-2}Q_{g^{-1}.X}(z),
\end{align}
for each $g = \left(\begin{smallmatrix}a & b \\ c & d \end{smallmatrix} \right) \in \PSL_2(\R)$. For brevity we write $X_{z}$ and $X_{z^\perp}$ for the orthogonal projection of $X$ to the positive line generated by $X(z)$ and to the negative plane $X(z)^\perp$ spanned by $U_1(z),U_2(z)$, respectively. We have the useful relations
\begin{align}\label{eq relations}
q(X_z) = \frac{1}{4}p_X(z)^2, \qquad q(X_{z^\perp}) = -\frac{1}{4}y^{-2}|Q_X(z)|^2.
\end{align}

Since $L$ has signature $(1,2)$, the Siegel theta function
\[
\Theta_{L}(\tau,z) = v \sum_{X \in L'}e\left(q(X_z)\tau + q(X_{z^\perp})\overline{\tau} \right)\e_X
\]
transforms like a modular form of weight $-1/2$ in $\tau$ for the Weil representation $\rho_L$ associated with $L$. Moreover, it is $\Gamma$-invariant in $z$. Here $\e_X = \e_{X + L}$ are the standard basis vectors of the group ring $\C[L'/L]$. The derivative $R_0\Theta_{L}(\tau,z) = 2i\frac{\partial}{\partial z}\Theta_L(\tau,z)$ transforms like a modular form of weight $2$ in $z$. It is explicitly given by
\begin{align}\label{eq raised theta}
R_0\Theta_L(\tau,z) = 2\pi v^2 \sum_{X \in L'}p_X(z) y^{-2}\overline{Q_X(z)}e\left(q(X_z)\tau + q(X_{z^\perp})\overline{\tau} \right)\e_X.
\end{align}

Let $A_{k,L}$ denote the space of all $\C[L'/L]$-valued functions on $\H$ which transform like modular forms of weight $k$ for $\rho_L$. If $M \subset L$ is a sublattice of finite index, we have natural maps $A_{k,L} \to A_{k,M}, f \mapsto f_M$, and $A_{k,M} \to A_{k,L}, g \mapsto g^L$, given by
\begin{align}\label{eq operators}
(f_{M})_\gamma = \begin{cases}f_{\gamma} & \text{if $\gamma \in L'/M$,} \\
0, & \text{if $\gamma \notin L'/M$,}
\end{cases} \qquad \text{and} \qquad (g^L)_\gamma = \sum_{\beta \in L/M}g_{\beta+\gamma}.
\end{align}
These maps are adjoint with respect to the bilinear pairings on $\C[L'/L]$ and $\C[M'/M]$. Moreover, the Siegel theta functions for $M$ and $L$ are related by $\Theta_L(\tau,z) = \Theta_M(\tau,z)^L$.

For a smooth modular form $f$ of weight $1/2$ for $\overline{\rho}_L$ the regularized theta lift is defined by
\[
\Phi_{L}(f,z) =  \int_{\mathcal{F}}^{\reg} f(\tau)\cdot \Theta_{L}(\tau,z) d\mu(\tau),
\]
where $\mathcal{F}$ is a fundamental domain for $\Gamma \backslash \H$, the dot denotes the natural bilinear pairing on $\C[L'/L]$, and $d\mu(\tau)$ is the invariant measure. The integral is regularized as explained in \cite{bruinierhabil}. 

We let $P_{3/2-k,d}(\tau)$ be the unique harmonic Maass form of weight $3/2-k$ for the dual Weil representation $\overline{\rho}_L$ with principal part $q^{d}\e_d+O(1)$.

\begin{proposition}[Proposition~4.1 in \cite{alfesbruinierschwagenscheidt}]\label{proposition theta lift} For odd $k\geq 3$ we have
\[
f_{k,d}(z) = c_k \cdot  R_{0}^k\Phi_{L}\left(R_{3/2-k}^{(k-1)/2}P_{3/2-k,d},z\right)
\]
with the constant $c_k = (2^{k+1}\pi^{(k+1)/2}(k-1)!)^{-1}$. Here $R_{\kappa}^n = R_{\kappa + 2n-2} \circ \dots \circ R_{\kappa}$ is an iterated version of the raising operator $R_{\kappa} = 2i\frac{\partial}{\partial z}+\kappa y^{-1}$.
\end{proposition}

\section{Cycle integrals of Siegel theta functions}\label{section cycle integral siegel theta}

 In this section we compute the cycle integrals of the derivative of the Siegel theta function $\Theta_L(\tau,z)$. Throughout this section we let
\[
A = \begin{pmatrix}b/2 & c \\ -a & -b/2 \end{pmatrix} \in L', \qquad D = -4q(A) = b^2-4ac > 0,
\]
be an indefinite integral binary quadratic form of positive non-square discriminant $D$. We let $S_A = \{z \in \H \, : \, a|z|^2 + bx + c = 0\}$ be the corresponding geodesic semi-circle in $\H$. Changing $A$ to $-A$ only changes the sign of the cycle integral, so we will assume $a > 0$ from now on. Then $S_A$ is oriented counter-clockwise.
Since $A$ is as a vector of negative norm in $L'$, we can define two sublattices
\[
I = L \cap (\Q A)^\perp, \qquad N = L \cap (\Q A),
\]
of signature $(1,1)$ and $(0,1)$, respectively. Note that $I$ is anisotropic since $A$ has non-square discriminant, and that $N^- = (N,-q)$ is a positive definite one-dimensional lattice. We define the weight $3/2$ unary theta function corresponding to $N$ (or rather $N^-$) by
\begin{align}\label{eq holomorphic theta function}
\Theta_{3/2,N}(\tau) = \sum_{W \in N'}(W,A)e(-q(W)\tau) \e_W.
\end{align}
It is a cusp form of weight $3/2$ for the dual Weil representation $\overline{\rho}_N$.

Let $w = \frac{-b-\sqrt{D}}{2a}$ be the real endpoint of the geodesic $S_A$, and consider the vector 
\begin{align}\label{eq Y0}
Y_0 = \begin{pmatrix} 
	-w & w^2 \\ -1 & w
\end{pmatrix} \in I \otimes \R.
\end{align}
We define the \emph{Hecke theta series} associated with the signature $(1,1)$ lattice $I$ by
\begin{align}\label{eq hecke theta series}
\vartheta_I(\tau) = \sum_{\substack{Y \in \Gamma_I \backslash I' \\ q(Y) > 0}}\frac{\sgn(Y,Y_0)}{|\Gamma_Y|}e(q(Y)\tau)\e_Y.
\end{align}
It is a cusp form of weight $1$ for $\rho_I$, which was first constructed by Hecke \cite{hecke} as a theta lift.

\begin{theorem}\label{theorem cycle integral siegel theta}
	We have
	\[
	\mathcal{C}_A\big(R_{0}\Theta_{L}(\tau,\,\cdot\,) \big) = -\frac{4\pi}{\sqrt{D}}\left(\vartheta_I(\tau) \otimes \overline{\Theta_{3/2,N}(\tau)}v^{3/2}\right)^L, 
	\]
	with the map $f^L$ defined in \eqref{eq operators}.
\end{theorem}

\begin{proof}
	Note that $I \oplus N$ is a sublattice of finite index in $L$, and we have 
\[
\mathcal{C}_A(R_{0}\Theta_{L}) = \mathcal{C}_A(R_{0}\Theta_{I \oplus N})^L,
\] 
	so we can assume $L = I \oplus N$ for now. We first split $R_{0}\Theta_{I \oplus N}$ along the geodesic $S_A$ into tensor products of theta functions associated with $I$ and $N$ . We can write $X \in (I\oplus N)'$ uniquely as $X = Y + W$ with $Y \in I'$ and $W \in N'$. Note that $W = \frac{(W,A)}{(A,A)}A$. Hence, for $z \in S_A$ we have
	\[
	p_W(z) = 0, \qquad W_z = 0, \qquad y^{-2}\overline{Q_W(z)} = -2(W,A) Q_A(z)^{-1}.
	\]
	Plugging this into the explicit formula \eqref{eq raised theta} for $R_{0}\Theta_{I \oplus N}$, we obtain for $z \in S_A$ the splitting
	\begin{align*}
	R_0\Theta_{I \oplus N}(\tau,z) &= 2\pi \left(\Theta_I^*(\tau,z)\otimes \overline{\Theta_{1/2,N}(\tau)}v^{1/2}\right) - 4\pi  \left(Q_A(z)^{-1}\Theta_I(\tau,z) \otimes \overline{\Theta_{3/2,N}(\tau)}v^{3/2}\right),
	\end{align*}
	with the Siegel theta functions
	\begin{align*}
	\Theta_I^*(\tau,z) &= v^{3/2}\sum_{Y \in I'}p_{Y}(z)y^{-2}\overline{Q_Y(z)}e(q(Y_{z})\tau + q(Y_{z^{\perp}})\overline{\tau})\e_{Y}, \\
	\Theta_I(\tau,z) &= v^{1/2}\sum_{Y \in I'}p_{Y}(z)e(q(Y_{z})\tau + q(Y_{z^{\perp}})\overline{\tau})\e_{Y},
	\end{align*}
	and the holomorphic weight $1/2$ and weight $3/2$ theta functions associated with $N$.
	
	It remains to compute the cycle integrals of $\Theta_I^*(\tau,z)$ and $Q_A(z)^{-1}\Theta_I(\tau,z)$. Note that the geodesic $S_A$ can be identified with the Grassmannian $\Gr(I)$ of positive lines in $I \otimes \R$. Hence, the cycle integrals can be computed as done by Hecke \cite{hecke}. We give the details for $\Theta_I(\tau,z)$.
	
	Recall that we assume $a > 0$. The two real endpoints $w < w'$ of $S_A$ are given by $w = \frac{-b-\sqrt{D}}{2a}$ and $w' = \frac{-b+\sqrt{D}}{2a}$. The matrix
\[
\sigma = \frac{a^{\frac{1}{2}}}{D^{\frac{1}{4}}}\begin{pmatrix}w' & w \\ 1 & 1 \end{pmatrix} \in \PSL_{2}(\R)
\]
maps $0$ to $w$ and $i\infty$ to $w'$ and hence maps the imaginary axis (oriented from $i\infty$ to $0$) to $S_{A}$ (oriented counter-clockwise). Note that $\sigma^{-1}. A = [0,-\sqrt{D},0]$. Modding out the stabilizer $\Gamma_A = \Gamma_I$ in the sum over $Y$, and using the parametrization $z = \sigma.it$ of $S_A$, we obtain
\[
\int_{\Gamma_A \backslash S_A}\Theta_I(\tau,z)\frac{dz}{Q_A(z)} = \frac{v^{1/2}}{\sqrt{D}}\int_{0}^\infty\sum_{Y \in \Gamma_I \backslash I'}\frac{1}{|\Gamma_Y|}p_{Y}(\sigma.it)e(q(Y_{\sigma.it})\tau + q(Y_{\sigma.it^{\perp}})\overline{\tau})\e_{Y}\frac{dt}{t}.
\]
	If we write $Y = \left(\begin{smallmatrix}\beta/2 & \gamma \\ -\alpha & -\beta/2 \end{smallmatrix} \right) \in I'$, a short compuation shows that
	\[
	\sigma^{-1}.Y = \begin{pmatrix}0 & \lambda \\ \lambda' & 0 \end{pmatrix}, \qquad \lambda = \frac{a}{\sqrt{D}}(\alpha w^2 + \beta w + \gamma) \in \Q\big(\sqrt{D}\big),
	\]
	where $\lambda'$ is the conjugate of $\lambda$ in $\Q(\sqrt{D})$. Using \eqref{eq invariances} and \eqref{eq relations} we find
	\begin{align*}
	p_{Y}(\sigma.it) &= p_{\sigma^{-1}.Y}(it) =  \lambda t^{-1}-\lambda' t, \\
	q(Y_{\sigma.it^\perp}) &= q((\sigma^{-1}.Y)_{it^\perp}) = -\frac{1}{4}(\lambda t^{-1}+\lambda' t)^2.
	\end{align*}
	Hence, using $q(Y_z)\tau + q(Y_{z^\perp})\overline{\tau} = q(Y)\tau - 2iv q(Y_{z^{\perp}})$, we get
	\begin{align*}
		\int_{\Gamma_A \backslash S_A}\Theta_I(\tau,z)\frac{dz}{Q_A(z)} = \frac{v^{1/2}}{\sqrt{D}}\sum_{Y \in \Gamma_I \backslash I'}\frac{1}{|\Gamma_Y|}\left(\int_{0}^{\infty}(\lambda t^{-1}-\lambda' t) e^{-\pi v (\lambda t^{-1} + \lambda' t)^2}\frac{dt}{t}\right)e(q(Y) \tau)\e_{Y}.
	\end{align*}
	A standard computation\footnote{For example, one can use the $K$-Bessel function $\sqrt{\frac{\pi}{2}}\frac{e^{-x}}{\sqrt{x}} = K_{1/2}(x) = \frac{1}{2}\sqrt{\frac{1}{2}z}\int_{0}^\infty \exp\left(-t - \frac{x^2}{4t}\right)\frac{dt}{t^{3/2}}$.} shows that the integral is given by
	\[
	\int_{0}^\infty (\lambda t^{-1}-\lambda' t) e^{-\pi v (\lambda t^{-1} + \lambda' t)^2}\frac{dt}{t} = \frac{1}{2}v^{-1/2}(\sgn(\lambda)-\sgn(\lambda'))e^{-2\pi v|\lambda\lambda'|}e^{-2\pi v \lambda\lambda'},
	\]
	which equals $v^{-1/2}\sgn(\lambda)$ if $q(Y)  = -\lambda\lambda' > 0$ and vanishes otherwise. Using $\sqrt{D}\lambda=a(\alpha w^2 + \beta w + \gamma) = a(Y,Y_0)$,	with $Y_0$ as in \eqref{eq Y0}, we get $\sgn(\lambda)=\sgn(Y,Y_0)$. Summarizing, we obtain
	\begin{align*}
	\int_{\Gamma_A \backslash S_A}\Theta_I(\tau,z)\frac{dz}{Q_A(z)}  &= \frac{1}{\sqrt{D}} \sum_{\substack{Y \in \Gamma_I \backslash I' \\ q(Y) > 0}}\frac{\sgn(Y,Y_0)}{|\Gamma_Y|}e(q(Y) \tau)\e_{Y} = \frac{1}{\sqrt{D}}\vartheta_I(\tau).
	\end{align*}
	
	The computation of the cycle integral of $\Theta_I^*(\tau,z)$ is very similar. Using
	\[
	\int_{0}^\infty (\lambda^2 t^{-2}-\lambda'^2 t^2)e^{-\pi v (\lambda t^{-1} + \lambda' t)^2} \frac{dt}{t} = \left[\frac{1}{2\pi v}e^{-\pi v (\lambda t^{-1} + \lambda' t)^2} \right]_0^{\infty} = 0,
	\]
	we see that we in fact have $\mathcal{C}_A(\Theta_I^*) = 0$. This finishes the proof.
\end{proof}

\begin{remark}
	One can compute the cycle integrals of $R_0 \Theta_{L}(\tau,z)$ along infinite geodesics (if $D$ is a square) in a similar way. In this case, the lattice $I$ is isotropic, and $\vartheta_I$ has to be replaced with a weight $1$ Hecke Eisenstein series.
\end{remark}

\section{Cycle integrals of meromorphic modular forms: The proof of Theorem~\ref{main theorem intro}}\label{section cycle integral meromorphic}

We are now ready to give an explicit evaluation of the cycle integrals appearing in Theorem~\ref{main theorem intro}. As before,  we let $A \in L'$ be an indefinite integral binary quadratic form of non-square discriminant $D = -4q(A) > 0$ with $a > 0$, and we let $I = L \cap (\Q A)^\perp$ and $N = L \cap (\Q A)$ be the corresponding indefinite and negative definite sublattices of $L$. Moreover, we let $\vartheta_I$ be the Hecke theta series of weight $1$ for $\rho_I$ defined in \eqref{eq hecke theta series} and $\Theta_{3/2,N}$ the holomorphic theta function of weight $3/2$ for $\overline{\rho}_{N}$ as in \eqref{eq holomorphic theta function}. 

We choose a harmonic Maass form $\widetilde{\Theta}_{3/2,N}$ of weight $1/2$ for $\rho_{N}$ with 
\[
\xi_{1/2}\widetilde{\Theta}_{3/2,N} = \frac{1}{\sqrt{D}}\Theta_{3/2,N},
\] 
and we let $\widetilde{\Theta}_{3/2,N}^+$ be its holomorphic part. Moreover, we require the Rankin-Cohen bracket
\[
\left[f,g \right]_{n} = \sum_{s = 0}^n (-1)^s \binom{\kappa+n-1}{s}\binom{\ell + n-1}{n-s}f^{(n-s)}g^{(s)}, \qquad \left(f^{(s)} = \frac{1}{(2\pi i )^s}\frac{\partial^s}{\partial \tau^s}f \right),
\]
 which maps modular forms $f$ of weight $\kappa$ and $g$ of weight $\ell$ to a form of weight $\kappa + \ell + 2n$.

\begin{theorem}\label{main theorem}
	Let $k \geq 3$ be odd and let $g \in M_{3/2-k,\overline{\rho}_L}^!$ be a weakly holomorphic modular form of weight $3/2-k$ for $\overline{\rho}_L$. Then we have
	\[
	\mathcal{C}_A\left(\sum_{d < 0}a_g(d)f_{k,d}\right) = -(4D)^{(k-1)/2}\CT\left(g_{I \oplus N} \cdot \left[\vartheta_I,\widetilde{\Theta}_{3/2,N}^+ \right]_{(k-1)/2} \right),
	\]
	with the Rankin-Cohen bracket in weights $\kappa = 1$ and $\ell = 1/2$, and $g_{I \oplus N}$ defined in \eqref{eq operators}.
\end{theorem}

\begin{proof} 
We follow the idea of the proof of \cite[Theorem~5.4]{bruinierehlenyang}. For brevity we put $f_{k,g} = \sum_{d < 0}a_g(d)f_{k,d}$. Using Proposition~\ref{proposition theta lift} we can write $f_{k,g}$ as a theta lift,
\[
f_{k,g}(z) = c_k\cdot R_{0}^{k}\Phi_L\left(R_{3/2-k}^{(k-1)/2}g,z\right),
\]
with the constant $c_k = (2^{k+1}\pi^{(k+1)/2}(k-1)!)^{-1}$. Now we take the cycle integral $\mathcal{C}_A$ on both sides. As a preliminary step, we will reduce the power of the outer raising operator $R_0^k$. It is well-known that $\Phi_L = \Phi_L(R_{3/2-k}^{(k-1)/2}g,z)$ is an eigenform of the Laplacian $\Delta_0$ with eigenvalue $k(1-k)$, see \cite[Theorem~4.6]{bruinierhabil}. Hence, a repeated application of \cite[Theorem~1.1]{alfesschwagenscheidtidentities} gives
\[
\mathcal{C}_A\big( R_{0}^k\Phi_L \big) = c_k'\, \mathcal{C}_A\left( R_{0}\Phi_L \right),
\]
with the constant $c_k' = D^{(k-1)/2} (k-1)!^2/(\frac{k-1}{2})!^2$. Interchanging the cycle integral and the regularized integral gives
\[
\mathcal{C}_A\big(f_{k,g}\big) = c_k c_k' \int_{\mathcal{F}}^{\reg}R_{3/2-k}^{(k-1)/2}g(\tau)\cdot \mathcal{C}_A\big(R_0\Theta_L(\tau,\, \cdot \,)\big)d\mu(\tau).
\]
Now we plug in the evaluation of the cycle integral from Theorem~\ref{theorem cycle integral siegel theta} to get
\[
\mathcal{C}_A\big(f_{k,g}\big) = c_k c_k'c_k''\int_{\mathcal{F}}^{\reg}R_{3/2-k}^{(k-1)/2}g_{I \oplus N}(\tau)\cdot \vartheta_I(\tau)\otimes \overline{\Theta_{3/2,N}(\tau)} v^{3/2}d\mu(\tau),
\]
with the constant $c_k'' = -4\pi/\sqrt{D}$. Here we used that the maps $f^L$ and $g_{I \oplus N}$ are adjoint to each other. Next, we use the ``self-adjointness'' of the raising operator to compute
\begin{align*}
\mathcal{C}_A\big(f_{k,g}\big) 
&= (-1)^{(k-1)/2}c_k c_k'c_k''\int_{\mathcal{F}}^{\reg}g_{I \oplus N}(\tau)\cdot R_{-1/2}^{(k-1)/2}\left(\vartheta_I(\tau)\otimes \overline{\Theta_{3/2,N}(\tau)} v^{3/2}\right)d\mu(\tau) \\
&= (-1)^{(k-1)/2}c_k c_k'c_k''\int_{\mathcal{F}}^{\reg}g_{I \oplus N}(\tau)\cdot \left(R_{1}^{(k-1)/2}\vartheta_I(\tau)\right)\otimes \overline{\Theta_{3/2,N}(\tau)} v^{3/2}d\mu(\tau).
\end{align*}
In the last step we used that $\Theta_{3/2,N}$ is holomorphic. By \cite[Proposition~3.6]{bruinierehlenyang} we have
\[
\left(R_{1}^{(k-1)/2}\vartheta_I(\tau)\right)\otimes \overline{\Theta_{3/2,N}(\tau)} v^{3/2} = c_k'''L_{k+1/2}\left[\vartheta_I, \widetilde{\Theta}_{3/2,N}\right]_{(k-1)/2},
\]
with the constant $c_k''' = 4^{k-1}(-\pi)^{(k-1)/2}\sqrt{D}(\frac{k-1}{2})!^2/(k-1)!$. We arrive at
\[
\mathcal{C}_A\big(f_{k,g}\big)  = (-1)^{(k-1)/2}c_k c_k'c_k''c_k'''\int_{\mathcal{F}}^{\reg}g_{I \oplus N}(\tau)\cdot  L_{k+1/2}\left[\vartheta_I, \widetilde{\Theta}_{3/2,N}\right]_{(k-1)/2} d\mu(\tau).
\]
Note that the constants in front simplify to $(-1)^{(k-1)/2}c_k c_k'c_k''c_k''' = -(4D)^{(k-1)/2}$. Now a standard application of Stokes' Theorem as in \cite[Theorem~5.4]{bruinierehlenyang} finishes the proof.
\end{proof}

We remark that there is a similar formula for the linear combination of cycle integrals of $f_{k,P}$ from Theorem~\ref{theorem abs}, see \cite[Theorem~1.1]{alfesbringmannmalesschwagenscheidt}. We can now prove Theorem~\ref{main theorem intro}.

\begin{proof}[Proof of Theorem~\ref{main theorem intro}]
	We show that the right-hand side in Theorem~\ref{main theorem} is rational. Note that $g$ and $\vartheta_I$ have rational Fourier coefficients. Moreover, by \cite[Theorem~1.1]{lischwagenscheidt} we can choose $\widetilde{\Theta}_{3/2,N}$ such that its holomorphic part has rational Fourier coefficients. Here we use that the lattice $N$ is of the form $(\Z,q(x) = -Dr^2x^2)$ for some rational number $r$. It remains to note that the Rankin-Cohen brackets preserve the rationality of the Fourier coefficients.
\end{proof}

\section{Generalization to higher level}\label{section generalizations}

In \cite[Conjecture~2.5]{loebrichschwagenscheidtcycleintegrals}, Theorem~\ref{main theorem intro} was stated as a conjecture for level $\Gamma_0(N)$ and twisted versions of $f_{k,d}$. Here we sketch the proof of this general conjecture.

Let $N \in \N$. We let $\Delta \in \Z$ be a fundamental discriminant (possibly $1$) and $\rho \in \Z/2N\Z$ such that $\Delta \equiv \rho^2 \pmod{2N}$. Moreover, we let $d < 0$ be a negative integer and $r \in \Z/2N\Z$ with $d \equiv \sgn(\Delta)r^2 \pmod{2N}$. We consider the set $\mathcal{Q}_{N,d|\Delta|,r\rho}$ of (not necessarily positive definite) integral binary quadratic forms $Q = [aN,b,c]$ of discriminant $d|\Delta|$ with $b \equiv r\rho \pmod{2N}$. We let $\chi_\Delta$ be the usual generalized genus character on $\mathcal{Q}_{N,d|\Delta|,r\rho}$. For $k \in \N$ with $k \geq 2$ we define the function
\[
f_{k,d,r,\Delta,\rho}(z) = \frac{|d\Delta|^{k-1/2}}{2\pi}\sum_{Q \in \mathcal{Q}_{N,d|\Delta|,r\rho}}\frac{\sgn(Q)\chi_\Delta(Q)}{Q(z,1)^k},
\]
where we put $\sgn(Q) = \sgn(a)$. It defines a meromorphic modular form of weight $2k$ for $\Gamma_0(N)$. For $N = 1$ and odd $k$ we recover $f_{k,d}$ as $f_{k,d,d,1,1}$.

We consider the signature $(1,2)$ lattice
\[
L_N = \left\{X=\begin{pmatrix}b & c/N \\ -a & -b \end{pmatrix}: a,b,c \in \Z \right\}
\]
with the quadratic form $q_N(X) = N\det(X)$. We have $L_N'/L_N \cong \Z/2N\Z$, so we will write $\e_r$ with $r \in \Z/2N\Z$ for the standard basis elements of $\C[L_N'/L_N]$. Note that the elements of $L_N'$ correspond to binary quadratic forms $[aN,b,c]$, with the discriminant being given by $-4Nq_N(X)$. We let $\widetilde{\rho}_{L_N}$ be equal to $\overline{\rho}_{L_N}$ if $\Delta > 0$, and to $\rho_{L_N}$ if $\Delta < 0$.

We have the following higher level version of Theorem~\ref{main theorem intro}.

\begin{theorem}
	Let $k \in \N$ with $k \geq 2$ and let $g$ be a weakly holomorphic modular form of weight $3/2-k$ for $\widetilde{\rho}_{L_N}$. Suppose that the Fourier coefficients $a_g(d,r)$ for $d < 0$ and $r \in \Z/2N\Z$ are rational. Then the cycle integrals
	\[
	\mathcal{C}_A\left(\sum_{r \in \Z/2N\Z}\sum_{d < 0}a_{g}(d,r)f_{k,d,r,\Delta,\rho}\right)
	\]
	are rational.
\end{theorem}

\begin{proof} This can be proved analogously to Theorem~\ref{main theorem intro}, so we only give a sketch. First, by \cite[Proposition~4.1]{alfesbruinierschwagenscheidt} the meromorphic modular form $f_{k,d,r,\Delta,\rho}$ can be written as a twisted theta lift on the lattice $L_N$. We can reduce to the non-twisted case using an intertwining operator for the involved Weil representations as explained in \cite{alfesehlen}. Then it remains to evaluate the cycle integrals of certain Siegel theta functions associated with a general even lattice $L$ of signature $(1,2)$. However, note that we did not use the precise shape of $L$ in the proof of Theorem~\ref{theorem cycle integral siegel theta}, so the arguments work exactly the same for arbitrary $L$, at least for odd $k$. For even $k$ the theta lift involves a Siegel theta function of the shape
\[
\Theta_{L}^*(\tau,z) = v\sum_{X \in L'}p_z(X)e\left(q(X_z)\tau + q(X_{z^\perp})\overline{\tau} \right)\e_X.
\]
A computation similar to the proof of Theorem~\ref{theorem cycle integral siegel theta} shows that we have
\[
\mathcal{C}_A\big(\Theta_{L}^*(\tau,z) \big) \stackrel{\cdot}{=} \left(\vartheta_I(\tau) \otimes \overline{\Theta_{1/2,N}(\tau)}v^{1/2}\right)^L,
\]
with the holomorphic weight $1/2$ theta function $\Theta_{1/2,N}(\tau) = \sum_{X \in N'}e(-q(X)\tau)\e_X$. Apart from this, the proofs for even and odd $k$ are analogous.
\end{proof}

Finally, one might speculate that the method from Theorem~\ref{theorem cycle integral siegel theta} could be used to compute cycle integrals of Siegel theta functions (and the corresponding theta lifts) on lattices of other signatures, e.g. signature $(1,n)$ or $(2,n)$. We hope to come back to this problem in the future.

\end{document}